\documentclass[11pt,a4paper,oneside, reqno]{amsart}

\usepackage[a4paper,inner=2.5cm,outer=2cm,top=3cm,bottom=3cm,pdftex]{geometry}

\usepackage{color}
\usepackage{amsmath,amssymb,amsfonts,amsthm}
\usepackage{graphicx}
\usepackage{enumerate}

 \newtheorem{theorem}{Theorem}[section]
 \newtheorem{corollary}[theorem]{Corollary}
 \newtheorem{lemma}[theorem]{Lemma}
 
 \theoremstyle{definition}

 \theoremstyle{conjecture}
 \newtheorem{conjecture}[theorem]{Conjecture}

 \theoremstyle{remark}
 \newtheorem{remark}[theorem]{Remark}

 \newcommand{\eps}{\varepsilon}

\newcommand{\norm}[1]{\big\Vert#1\big\Vert}

\newcommand{\abs}[1]{\left\vert#1\right\vert}
\newcommand{\set}[1]{\left\{\,#1\,\right\}}

\newcommand{\inner}[1]{\left(#1\right)}

\newcommand{\comi}[1]{\left<#1\right>}
\newcommand{\com}[1]{\left[#1\right]}
\newcommand{\reff}[1]{(\ref{#1})}

\begin{document}

\title[Compact resolvent of the Fokker-Planck operator]{Compactness Criteria for the Resolvent of  the Fokker-Planck operator}%

\author[ W.-X. Li]{ Wei-Xi Li}

\date{}

\address{\newline 
 Wei-Xi Li,
School of Mathematics and Statistics, and Computational Science Hubei Key Laboratory,   Wuhan University,  430072 Wuhan, China\\
\newline
\and
\newline
The Institute of Mathematical Sciences, The Chinese University of Hong Kong, Shatin, NT,
Hong Kong}

\email{
wei-xi.li@whu.edu.cn}

\begin{abstract}
 In this paper we study the spectral  property of a Fokker-Planck operator with potential.  By virtue of  a multiplier method inspired by  Nicolas Lerner,  we  obtain  new compactness criteria  for  its resolvent,  involving  the control of the positive eigenvalues of  the Hessian matrix of the  potential.     \end{abstract}



\maketitle

\section{Introduction and main results}

The  Fokker-Planck operator reads
\begin{eqnarray}\label{FP++}
  P=y\cdot\partial_x-\partial_x V(x)\cdot
  \partial_y-\Delta_y+\frac{\abs y^2}{4}-\frac{n}{2},\quad
  (x,y)\in\mathbb{R}^{2n},
\end{eqnarray}
where $x$ denotes the space variable and $y$ denotes the velocity
variable, and $V(x)$ is a potential defined in the whole spatial space
$\mathbb{R}_x^n$.   In this work we are mainly concerned with the compact resolvent property for the non-selfadjoint Fokker-Planck operator, and this is motivated by 
 a conjecture stated by Helffer and Nier (see \cite[Conjecture 1.2]{HelfferNier05}), which reveals the close
 link between the compact resolvent property for the Fokker-Planck operator and the same property for the corresponding Witten Laplacian. 
 Precisely, 

\begin{conjecture}[Helffer-Nier's Conjecture]\label{hnc} 
 The Fokker-Planck operator $P$ has a compact resolvent if and
only if the Witten Laplacian $\Delta_{V/2}^{(0)}$,   defined by
 \begin{eqnarray*}
   \Delta_{V/2}^{(0)}=
  -\Delta_x+\frac{1}{4}\abs{\partial_xV(x)}^2-\frac{1}{2}\Delta_x
  V(x),	
 \end{eqnarray*}
 has a compact  resolvent. 	
\end{conjecture}

The necessity part, that the Witten
Laplacian $\Delta_{V/2}^{(0)}$ has a compact resolvent  if the
Fokker-Planck operator $P$ is with compact resolvent,  has already established by Helffer and Nier (c.f. \cite[Theorem 1.1]{HelfferNier05}). The reverse implication   still remains open up to now for general potential,  and  it is indeed valid under some conditions on the potential $V.$  For instance,  following  the analysis in \cite{HelfferNier05, HerauNier04} with some improvements,  
  the author (\cite{lipisa})  proved  that  if  $V$ satisfies that   
\begin{equation}\label{+HyP}
   \forall~\abs\alpha=2,~\exists~C_\alpha>0,\quad \abs{\partial_x ^\alpha V(x)}\leq C_\alpha \comi{\partial_x V(x)}^{s}
    ~~\,\,{\rm with~~s<\frac{4}{3}},
\end{equation} 
then Fokker-Planck operator has a compact resolvent provided the Witten-Laplacian has a compact resolvent or $\lim_{\abs x\rightarrow +\infty} \abs{\partial_xV(x)}=+\infty,$ and moreover a constant $C$ exists such that the following weighted  estimate
\begin{eqnarray*} 
	\norm{\abs{\partial_xV(x)}^{2/3} u}_{L^2}\leq C\inner{\norm{P u}_{L^2}+\norm{u}_{L^2}}
\end{eqnarray*}   
holds for all $u\in C_0^\infty\inner{\mathbb{R}^{2n}}$.   Here and
throughout the paper we will use the notation
\begin{eqnarray*}
	\comi{\cdot}=\inner{1+\abs{\cdot}^2}^{1/2},
\end{eqnarray*}
which is  equivalent to   the Modulus  $\abs{\cdot}$,
and use
 $\Vert \cdot\Vert_{L^2}$ and $\comi{\cdot,\cdot}_{L^2}$ to denote respectively 
the norm  and inner product of the complex Hilbert space
$L^2\big(\mathbb{R}^{2n}\big),$ and denote by $
C_0^\infty\inner{\mathbb{R}^{2n}}$ the set of smooth compactly
supported functions.

We remark  the  drawback of the condition \eqref{+HyP} is that it doesn't give any information for the dependence on the sign of $V$, which plays import role in the analysis of compact resolvent property for Witten Laplacian.   For instance   it is well-known  (see \cite{HelfferNier05,Nier-BJ}) that the Witten Laplacian $\Delta_{V/2}^{(0)}$ with  $V=-x_1^2x_2^2$  has a compact resolvent, while  $0$ actually belongs to the essential spectrum of  Witten Laplacian $\Delta_{V/2}^{(0)}$ with $V= x_1^2x_2^2$ and thus its  resolvent cannot be compact.     By the  general criteria for  Schr\"odinger operators  we  see   if 
\begin{eqnarray*}
	\frac{1}{4}\abs{\partial_xV(x)}^2-\frac{1}{2}\Delta_x V\rightarrow+\infty,   ~~{\rm as} ~~\abs x\rightarrow+\infty,
\end{eqnarray*} 
or more generally (see \cite[Proposition 3.1]{HelfferNier05} for instance), if
 \begin{eqnarray}\label{tv}
	\frac{t}{4}\abs{\partial_xV(x)}^2-\frac{1}{2}\Delta_x V\rightarrow+\infty,   ~~{\rm as} ~~\abs x\rightarrow+\infty
\end{eqnarray}
for some $t\in]0,2[$,  then 
 the  Witten Laplacian $ \Delta_{V/2}^{(0)}$ has a compact resolvent.   We refer the reader to \cite{HelfferNier05}  for other criteria presented with detailed discussion.     These criteria show  the microlocal property, i.e., the dependence on the sign of $V$,  for the compact resolvent of Witten Laplacian. 
    As far as   Fokker-Planck operator  is  concerned,  Helffer-Nier's Conjecture suggests strongly  it should have the similar microlocal property  as the Witten Laplacian.   And this kind of dependence property for Fokker-Planck operator is not clear by now.    In the present work we will give some sufficient conditions for the compact resolvent of Fokker-Planck operator, mainly based on the sign of the eigenvalues of the  Hessian matrix  $\inner{\partial_{x_ix_j}V}_{1\leq i,j\leq n}$. 
    Our  results can be stated as follows.

\begin{theorem}\label{mtm+}
 Denote by $\lambda_\ell(x), 1\leq \ell\leq n$,  the eigenvalues of the Hessian  matrix   
    \begin{eqnarray*}
 \inner{\partial_{x_ix_j}V(x)}_{1\leq i, j\leq n}. 
 \end{eqnarray*}
With each $x\in\mathbb R^n$ we associate a set $I_{x}$ of indexes  defined by 
    \begin{eqnarray*}
  	I_{ x}= \Big\{1\leq \ell \leq n;~\lambda_\ell(x)>0\Big\}. 
  \end{eqnarray*}
    Suppose  that there exists a constant $C$ such that 
    \begin{eqnarray}\label{assext}
    \forall~x\in\mathbb R^n,\quad	\sum_{j\in I_x} \lambda_j(x)\leq  C  \comi{\partial_x
  V(x)}^{4/3}.
    \end{eqnarray} 	
Then the following conclusions hold. 
\begin{enumerate}[(i)]
\item There exists a constant $C_*$ such that 
\begin{eqnarray*}
\forall u\in C_0^\infty\inner{\mathbb{R}^{2n}},\quad	\norm{\abs{\partial_xV(x)}^{1/16} u}_{L^2}\leq C_*\inner{\norm{P u}_{L^2}+\norm{u}_{L^2}}.
\end{eqnarray*}   
As a result, the Fokker-Planck operator $P$ has a compact resolvent if 
\begin{eqnarray*}
	\lim_{\abs x\rightarrow +\infty} \abs{\partial_x V(x)}=+\infty. 
\end{eqnarray*}
\item  Suppose   there exists a number  $\alpha\geq 0$, such that 
\begin{eqnarray}\label{longtimebe}
	\lim_{\abs x\rightarrow +\infty} \Big(\alpha\abs{\partial_x V(x)}^2-\Delta_x V(x)\Big)=+\infty.
\end{eqnarray}
   Then we can find  a   constant $\tilde C$,  depending on $\alpha,$   such that 
\begin{eqnarray*}
\forall u\in C_0^\infty\inner{\mathbb{R}^{2n}},\quad		 \norm{\big|\alpha \abs{ \partial_x V(x)}^2-\Delta_x V(x)\big|^{1/ 80}  u }_{L^2} 
		  \leq \tilde C \inner{ \norm{ P u}_{L^2}+\norm{u}_{L^2}}, \end{eqnarray*}   
and thus      
the Fokker-Planck operator  $P$  has a compact resolvent as a result. 
\end{enumerate}  
\end{theorem}

The assumption \eqref{assext} is an improvement of the condition \eqref{+HyP}.  We mention that the index $4/3$ in  \eqref{assext}   is not sharp,  and the  following Theorem \ref{+Hypo}  and Corollary  \ref{+++++++corone++} are devoted to showing   a better index $14/5$ may be expected. 

\begin{theorem}\label{+Hypo}
Suppose that there exists a number  $\tau\geq 0$,  such that  the  matrix 
$$
A_\tau(x)=\inner{a_{ij}^{\tau}(x)}_{1\leq i, j\leq n},\qquad	a_{ij}^{\tau}=\tau\comi{\partial_x V}^{4\over 5} \inner{\partial_{x_i}V}\inner{\partial_{x_j}V}- \partial_{x_ix_j}V+\tau \delta_{ij}
$$
 is positive-definite for all $x\in \mathbb R^n$,  where $\delta_{ij }$ is the Kronecker Delta.   Then there is a constant $C,$ such that 
  \begin{eqnarray}\label{+A1}
   \forall u\in
 C_0^\infty\big(\mathbb{R}^{2}\big),\quad 	\norm{ \abs{\partial_x V(x)}^{1/20}u}_{L^2}+ \sum_{1\leq i,j\leq n}\norm{\abs{a_{ij}^\tau(x)}^{1/80}u}_{L^2}\leq C\Big(  \norm{P u}_{L^2}+\norm{ u}_{L^2}\Big).
\end{eqnarray}
As a result,  the Fokker-Planck operator $P$ has a compact resolvent if   
  \begin{eqnarray*}
  	\lim_{\abs x\rightarrow+\infty} \Big(\abs{\partial_x V(x)}+\sum_{1\leq i,j\leq n}\abs{a_{ij}^\tau(x)}\Big)=+\infty.
  \end{eqnarray*}
\end{theorem}

As an immediate consequence,  when $n=1$ we have the  compactness criteria  for Fokker-Planck operator,   which is an improvement of  the corresponding  condition \eqref{tv}  for Witten Laplacian.  Precisely,    

\begin{corollary}\label{+++++++corone++}
	Let $n=1$ and let $V(x)\in C^2(\mathbb{R})$.  Suppose that there exists $\tau\geq 0,$ such that 
	\begin{eqnarray*}
		\lim_{\abs x\rightarrow+\infty}  \tau \abs{\partial_x V(x)}^{14/5}-\Delta_x V(x)=+\infty.
	\end{eqnarray*}
	Then the Fokker-Planck operator  $P$ has a compact resolvent. 
\end{corollary}

\begin{remark}
In  the special case  when $n=1$,   using  Corollary \ref{+++++++corone++} and  the  necessity part in Helffer-Nier Conjecture  (c.f. \cite[Theorem 1.1]{HelfferNier05}),  we can also improve   the criteria \eqref{tv} for Witten Laplacian,   by  allowing $t$  to  range over $[0,+\infty[$ instead of $]0,2[$   and  relaxing  the index $2$ there by $14/5$.	
\end{remark}

\begin{remark}
The hypotheses  in Theorem \ref{mtm+} and Theorem \ref{+Hypo},  are related to  the sign of the eigenvalues of the  Hessian matrix  $\inner{\partial_{x_ix_j}V}_{1\leq i,j\leq n}.$  In fact these assumptions are obviously  fulfilled when the Hessian matrix is negative-semidefinite.    When the Hessian matrix is positive-semidefinite  or  indefinite,  we requires that   the   {\bf positive} eigenvalues of Hessian matrix,  instead of  all the second derivatives in the condition \eqref{+HyP},  are dominated by   $\comi{\partial_x V}^{4/3}$. 
Now look back at the aforementioned potential  $V=\pm x_1^2x_2^2$, and it is clear that these hypotheses are fulfilled by    $V=- x_1^2x_2^2$  and  violated by    $V=x_1^2x_2^2$. 
\end{remark}

\begin{remark}
In \cite{HelfferNier05,Hel-Nou},  the authors introduced a compactness criterion for Witten Laplacian with polynomial potential $V$,  based on the group theory.  And it is also natural and  interesting to expect the similar  group theoretical compactness  criteria  for Fokker-Planck operator.     Now  consider such  a  potential,   not necessary to be a  polynomial,   that the matrix 
  \begin{eqnarray*}
\tilde A_\tau(x)=\inner{\tilde a_{ij}^{\tau}(x)}_{1\leq i, j\leq n}	,\quad \tilde a_{ij}^{\tau}(x)=\tau \abs{\partial_xV(x)}^{4/5}\inner{\partial_{x_i}V(x)}\inner{\partial_{x_j}V(x)}- \partial_{x_ix_j}V(x)
 \end{eqnarray*}
 is positive-semidefinite for some $\tau\geq 0.$   This condition is slightly stronger than the one in Theorem \ref{+Hypo},  and it  yields
 \begin{eqnarray*}
 	-\Delta_x V(x)+\tau \abs{\partial_x V(x)}^{14/5}\geq 0. 
 \end{eqnarray*}  
 Thus repeating the arguments used to prove maximum principle for elliptic equations, we see $V$ 
 doesn't have local minimum in $\mathbb R_x^n,$ except the constant-valued potentials.   So this kind of microlocal  property is imposed directly on the potential  rather than its "limiting polynomials" in the sense of \cite{HelfferNier05, Hel-Nou}.  
 \end{remark}

 Due to the lack of  estimates on the higher derivatives of $V$,  we can't follow the global symbolic calculus to prove our results,  although this method is efficiently explored to investigate the hypoellipticity and the compact resolvent  of Fokker-Planck operator (c.f. \cite{HelfferNier05, HerauNier04}).   Instead  we will use a multiplier method inspired by N. Lerner (see for instance \cite{Hormander85, MR2599384} and references therein),  which is based on the Poisson bracket analysis for the real and imaginary parts of the Fokker-Planck operator.        
 We hope this method not only applies to analyze the weighted estimate and the compact resolvent,   but also may 
 give insights on the sign conditions to investigate the subellipticity  (see \cite{Bouchut02, chenlixu, HelfferNier05, HerauNier04, lipisa, Nier-GFP} for instance) of Fokker-Planck operator.

 We end up the introduction  by  mentioning  that as a diffusive models,    the study of Fokker-Planck equation is  of independent interest in  kinetic theory and nonequilibrium statistical physics.   Here one of the basic problems is to analyze the large time behavior of solutions  to the time-dependent Fokker-Planck equation and prove  that these solutions converge  exponentially towards the equilibrium as $t $ goes to  $+\infty$.    Various approaches, such as hypoellipticity,  hypocoercivity, entropy method and so on,   are developed to study this problem, and  satisfactory results  are achieved.  We refer to \cite{DesvVillani01, HelfferNier05, Herau15, Herau07, Herau06, HerauNier04, lipisa,  villani}  and   references therein for more detail and \cite{AH} for the spectral analysis on the  non-selfadjoint Schr\"odinger operators with compact resolvent.   Finally we  remark that in order to study the   exponentially trend problem,  an efficient method  is  to  investigate the spectral gap, which is usually reduced to  analyze the compactness of resolvent.  On the other hand,   when the Fokker-Planck operator has an essential spectrum,   only polynomial convergence rate is expected, see the recent work \cite{wang15} for the study  on short-range potentials.

 \bigskip
\noindent{\bf Acknowledgements.}  We are grateful to the referee for the valuable suggestions which helped the author correct several mistakes.       Part of the work was done when the author visited The Institute of Mathematical Sciences,  The Chinese University of Hong Kong,  and he would like to thank their support and  hospitality.  We also appreciate the support from NSF of China(11422106) and Fok Ying Tung Education Foundation (151001).

\section{Proof the main results}\label{sec2}
We  firstly list some notations and facts used throughout the paper.   The proofs of the main results, Theorem \ref{mtm+} and Theorem \ref{+Hypo}, are presented in Subsection \ref{sec22}  and   Subsection \ref{sub21}, where  two multipliers $\mathcal M$ and $\mathcal K$ (see Lemma \ref{prpkey} and Lemma \ref{prpkey+} below) are introduced respectively.  This kind of multiplier method is inspired by N. Lerner \cite{Lernerpersonal}, and here it means that we have  to  choose carefully  an    operator  $M$ ({\it multiplier})  which is bounded  and self-adjoint in $L^2$ space, such that 
\begin{eqnarray*}
	{\rm Re}\, \Big<\big(y\cdot\partial_x-\partial_x V(x)\cdot
  \partial_y\big)u,~M u\Big>_{L^2}
\end{eqnarray*}
has  a good lower bound (weighted estimate here) on one side, and on the other side,   
\begin{eqnarray*}
	\abs{ \Big<\inner{-\Delta_y+ \abs y^2/4-n/2}u,~M u\Big>_{L^2}} 
\end{eqnarray*}
is    bounded from above by $\norm{Pu}_{L^2}+\norm{u}_{L^2}$.   
The multipliers chosen here are motivated by the Poisson bracket analysis for the real and imaginary parts of  symbol for the  Fokker-Planck operator.    Precisely,  if we denoted by $[Q_1,~Q_2]$ the commutator between  two operators $Q_1$ and $Q_2$, which is defined by 
\begin{eqnarray*}
	\com{Q_1,~Q_2}=Q_1Q_2-Q_2Q_1, 
\end{eqnarray*}
and also   use the notation that     
$X_0=y \cdot \partial_x-\partial_x V(x)\cdot
 \partial_y$ and $X_j=\partial_{y_j}+\frac{y_j}{2},j=1,\cdots n,$  
 then
 we can rewrite the Fokker-Planck operator $P$ define in \reff{FP++} as
\begin{equation*} 
  P=X_0+\sum_{j=1}^nX_j^*X_j,
\end{equation*}
and moreover we compute    
\begin{eqnarray}\label{com+first}
	\Big[ X_0,~ \sum_{j=1}^nX_j^*X_j\Big]=-\frac{1}{2} \partial_x V\cdot y+2\partial_{x}\cdot\partial_y
\end{eqnarray} 
and
\begin{eqnarray*}
	\Big[X_0,~ \Big[ X_0,~ \sum_{j=1}^nX_j^*X_j\Big]\,\Big]=-2\Delta_x+\frac{1}{2}\abs{\partial_xV}^2 -\frac{1}{2} \sum_{1\leq i,j\leq n}\inner{\partial_{x_ix_j} V} y_iy_j+2\sum_{1\leq i,j\leq n}\inner{\partial_{x_ix_j} V} \partial_{y_i}\partial_{y_j}.
\end{eqnarray*} 
Thus   the properties of subelliptic and weighted estimates in $x$ variable can be deduced  from the commutator above if some kind of conditions (negative semi-definite for instance) are imposed on the Hessian matrix $\inner{\partial_{x_ix_j}V}_{1\leq i,j\leq n}$.  This suggests that the multipliers  $\mathcal M$ and $\mathcal K$    here (see Lemma \ref{prpkey} and Lemma \ref{prpkey+} below)  should be chosen  through the first-commutator analysis in \eqref{com+first}.   In this work we will investigate only the  weighted estimate and thus the essential part in the multipliers is  the term $\partial_x V\cdot y$ in \eqref{com+first}. Moreover it seems reasonable that the term $\partial_x \cdot \partial_y$ in \eqref{com+first} shouldn't  be involved  in the multipliers, since it corresponds to the hypoellipticity and thus more estimates on the higher derivatives of $V$ are required rather than the ones of second order. 
  We refer to \cite{ahl, hlkyto} for the multipliers  introduced to deduce the hypoellipticity of kinetic operators.

Next we will give some estimates to be used frequently.  Observe $X_0$ is an anti-selfadjoint operator in $L^2$ and thus it is clear that 
\begin{equation}\label{revel}
\forall~u\in C_0^\infty(\mathbb R^{2n}),\quad 	\sum_{j=1} ^n\norm{ X_j u}_{L^2}\leq \norm{ \comi y u}_{L^2}+ \sum_{1\leq j\leq n}\norm{ \partial_{y_j} u}_{L^2}\leq C \inner{ \norm{ P u}_{L^2}+\norm{u}_{L^2}}.
\end{equation}
   We will use  the following result which is just a  consequence of  H\"ormander's bracket condition (cf. \cite{B0CaNo} for instance),  i.e.,    a constant $C$ exists  such that for any  vector-valued function $\theta(x)=(\theta_1(x),\cdots, \theta_n(x))$ of $x$ variable and for any $v \in  C_0^\infty(\mathbb R^{n}_y)$, we have
 \begin{eqnarray} \label{comest}
\norm{\abs{\theta_k(x)}^{1/2}v}_{L^2(\mathbb R_y^n)}  \leq C \Big( \norm{ (\theta(x) \cdot y)  v}_{L^2(\mathbb R_y^n)}+\norm{\partial_{y_k} v}_{L^2(\mathbb R_y^n)}\Big).
 \end{eqnarray}

\subsection{Proof of Theorem \ref{mtm+}}\label{sec22} 

We prove in this subsection  Theorem \ref{mtm+}.  To do so we begin with the following estimate  which holds for quite general potential.  

\begin{lemma}\label{prpkey}
Let $V(x)\in C^2\inner{\mathbb R^n}$. Then for all $\sigma\in\, ]0,1[,$ there exists a constant $C_\sigma>0$ such that for any $u\in C_0^\infty(\mathbb R^{2n})$  we have
 	\begin{eqnarray*}
  	&& \sigma \Big<\frac{(1+\abs y^8) \comi{\partial_x
  V(x)}^2  }{\rho ^3} \frac{y^2}{1+y^2}u,~u\Big>_{L^2}+ \Big<\frac{\rho }{\comi y^4}u,~u\Big>_{L^2}  \\
  &&\qquad\qquad\qquad - \sum_{1\leq i,j\leq n}\Big<\frac{(1+\abs y^8) \inner{\partial_{x_ix_j}V} y_iy_j}{\rho ^3}\frac{y^2}{1+y^2}u,~u\Big>_{L^2}\\
  &\leq&  C_\sigma \inner{ \norm{ P u}_{L^2}^2+\norm{u}_{L^2}^2}, 
  \end{eqnarray*}
 where the function $\rho\in C^1(\mathbb R^{2n})$ is defined by
 \begin{eqnarray*}
	\rho =\rho (x,y)=\inner{1+\abs y^{8}+\abs{\partial_xV(x)\cdot y}^2}^{1\over 2}.
\end{eqnarray*}
\end{lemma}

\begin{proof} To simplify the notations we will use $C$  in the proof to denote different constants, and similarly  use $C_\eps$ to denote different constants depending on $\eps.$     This lemma is to be proven by the multiplier method.  Firstly we introduce a multiplier $\mathcal M$, which is a $C^1(\mathbb R^{2n})$ function defined by 
  \begin{eqnarray*}
	\mathcal M=\mathcal M (x,y)= \frac{2 \partial_xV(x)\cdot y }{\rho (x,y)  } \frac{y^2}{1+y^2},
\end{eqnarray*} 
with  $\rho$  given in Lemma \ref{prpkey}. 	
Recall $X_0=y \cdot \partial_x-\partial_x V(x)\cdot
 \partial_y$.  Then  using the relation
\begin{eqnarray*}
	\frac{1}{2}\big[\mathcal M ,~X_0\big]=\frac{ \partial_xV(x)\cdot y }{\rho (x,y)}\Big[ \frac{y^2}{1+y^2},~X_0\Big]+ \Big[ \frac{ \partial_xV(x)\cdot y }{\rho (x,y)},~X_0 \Big]\frac{y^2}{1+y^2},
\end{eqnarray*}
we calculate
\begin{eqnarray*}
	\frac{1}{2}\big[\mathcal M ,~X_0\big]
	&=&\frac{2\abs{\partial_xV(x)\cdot y}^2}{\rho (x,y)(1+y^2)^2}+  \frac{  (1+\abs y^8) \abs{\partial_x
  V(x)}^2  }{\rho ^3} \frac{y^2}{1+y^2}\\
 &&-  \frac{  (1+\abs y^8) \sum_{1\leq i,j\leq n}\inner{\partial_{x_ix_j}V} y_iy_j}{\rho ^3}\frac{y^2}{1+y^2}\\
 &&- \frac{  4\abs{\partial_{x}V(x)\cdot y}^2 \abs y^6}{\rho ^3}\frac{y^2}{1+y^2} .
\end{eqnarray*}
As a result,  observe 
\begin{eqnarray*}
	{\rm Re}\comi{X_0 u,~\mathcal M u}_{L^2}=
  \frac{1}{2}\comi{\com{\mathcal M ,~X_0}u,~u}_{L^2},
\end{eqnarray*}
and thus
 \begin{eqnarray*}
  {\rm Re}\comi{X_0 u,~\mathcal M u}_{L^2}&=&
  \Big<\frac{(1+\abs y^8) \abs{\partial_x
  V(x)}^2 }{\rho ^3} \frac{y^2}{1+y^2}u,~u\Big>_{L^2}+ \Big<\frac{2\abs{\partial_xV(x)\cdot y}^2}{\rho (x,y)(1+y^2)^2}u,~u\Big>_{L^2}\\
  &&- \Big<\frac{(1+\abs y^8)  \sum_{1\leq i,j\leq n}\inner{\partial_{x_ix_j}V} y_iy_j}{\rho ^3}\frac{y^2}{1+y^2}u,~u\Big>_{L^2}\\
  &&- \Big<\frac{  4\abs{\partial_{x}V(x)\cdot y}^2 \abs y^6}{\rho ^3}\frac{y^2}{1+y^2}u,~u\Big>_{L^2}.\end{eqnarray*}
  This,  along with the inequalities  that
  \begin{eqnarray*}
  	\Big<\frac{2\rho }{(1+y^2)^2}u,~u\Big>_{L^2}&=&	 \Big<\frac{2\abs{\partial_xV(x)\cdot y}^2}{\rho (x,y)(1+y^2)^2}u,~u\Big>_{L^2}+ \Big<\frac{2(1+\abs y^8)}{\rho (x,y)(1+y^2)^2}u,~u\Big>_{L^2}\\
  	&\leq &	 \Big<\frac{2\abs{\partial_xV(x)\cdot y}^2}{\rho (x,y)(1+y^2)^2}u,~u\Big>_{L^2}+ C \norm{  u}_{L^2}^2
  \end{eqnarray*}
  and
  \begin{eqnarray*}
  	\abs{\Big<\frac{  4\abs{\partial_{x}V(x)\cdot y}^2 \abs y^6}{\rho ^3}\frac{y^2}{1+y^2}u,~u\Big>_{L^2}} \leq C \norm{  \comi y u}_{L^2}^2\leq C \Big(\norm{  P u}_{L^2}^2+\norm{ u}_{L^2}^2\Big), 
  \end{eqnarray*} 
  yields
    \begin{eqnarray}\label{est1+++}
 && \Big<\frac{ (1+\abs y^8) \abs{\partial_x
  V(x)}^2 }{\rho ^3} \frac{y^2}{1+y^2}u,~u\Big>_{L^2}+ \Big<\frac{2\rho }{(1+y^2)^2}u,~u\Big>_{L^2}\nonumber \\
  &&\qquad\qquad- \Big<\frac{(1+\abs y^8)  \sum_{1\leq i,j\leq n}\inner{\partial_{x_ix_j}V} y_iy_j}{\rho ^3}\frac{y^2}{1+y^2}u,~u\Big>_{L^2}\\
  &\leq&  {\rm Re}\comi{X_0 u,~\mathcal M u}_{L^2}+ C\inner{ \norm{ P u}_{L^2}^2+\norm{u}_{L^2}^2}.  \nonumber 
  \end{eqnarray}
On the other hand,  since $\mathcal M \in L^\infty(\mathbb R^{2n})$ with $\norm{\mathcal M }_{L^\infty}\leq 2$ then
 it is easy to see 
 \begin{equation}\label{pumu}
 	\abs{{\rm Re}\comi{P u,~\mathcal M u}_{L^2}}\leq 
  \norm{Pu}_{L^2}^2+ \norm{u}_{L^2}^2.
 \end{equation} 
Recall $X_j=\partial_{y_j}+\frac{y_j}{2},j=1,\cdots n.$ Then direct computation gives,    for each $1\leq j\leq n$, \begin{eqnarray*}
 \com{\mathcal M ,~X_j}&=&-\frac{2\inner{1+\abs y^8} \inner{\partial_{x_j} V(x) } }{\rho ^3 }\frac{y^2}{(1+y^2)}+\frac{8\inner{\partial_{x} V(x) \cdot y }\abs y^6 y_j}{\rho  ^3}\frac{y^2}{(1+y^2)}\\
 &&-\frac{4\inner{\partial_{x} V(x) \cdot y }y_j}{\rho(x,y)  (1+y^2)^2}.
\end{eqnarray*} 
Thus, for any $\eps>0,$
\begin{eqnarray*}
	&&\abs{ {\rm Re} \comi{ X_j u,~\com{\mathcal M ,~ X_j} u}_{L^2}	} \\
	&\leq &    \eps \norm{\rho^{-3}(1+y^2)^{-1}   (1+\abs y^8)y^2 \abs{\partial_{x_j} V(x) } u}_{L^2}^2	  +C_ \eps \inner{ \norm{X_j u}_{L^2}^2+\norm{u}_{L^2}^2}\\	
	&\leq &\eps  \Big<\frac{ (1+\abs y^8)\abs{\partial_{x_j}
  V(x)}^2 }{\rho ^3} \frac{y^2}{1+y^2}u,~u\Big>_{L^2}+C_ \eps \inner{ \norm{ X_j u}_{L^2}^2+\norm{u}_{L^2}^2}.
\end{eqnarray*}
This gives, using  $\norm{\mathcal M }_{L^\infty}\leq 2$ and \eqref{revel},  
\begin{eqnarray*}
\abs{ {\rm Re} \comi{X_j^*X_j u,~\mathcal M   u}_{L^2}	} &\leq & \abs{{\rm Re} \comi{ X_j u,~\mathcal M  X_j u}_{L^2}	}+ \abs{ {\rm Re} \comi{ X_j u,~\com{\mathcal M ,~ X_j} u}_{L^2}	}\\
&\leq &\eps  \Big<\frac{ (1+\abs y^8) \abs{\partial_{x_j}
  V(x)}^2 }{\rho ^3} \frac{y^2}{1+y^2}u,~u\Big>_{L^2}+C_ \eps \inner{ \norm{ X_j u}_{L^2}^2+\norm{u}_{L^2}^2}\\
&\leq& \eps  \Big<\frac{ (1+\abs y^8)\abs{\partial_{x_j}
  V(x)}^2 }{\rho ^3} \frac{y^2}{1+y^2}u,~u\Big>_{L^2}+C_ \eps \inner{ \norm{ P u}_{L^2}^2+\norm{u}_{L^2}^2}.
\end{eqnarray*}
Then
\begin{equation*}
	\sum_{j=1}^n \abs{ {\rm Re} \comi{X_j^*X_j u,~\mathcal M   u}_{L^2}	}\leq  \eps  \Big<\frac{(1+\abs y^8) \abs{\partial_{x}
  V(x)}^2  }{\rho ^3} \frac{y^2}{1+y^2}u,~u\Big>_{L^2}+C_\eps \inner{ \norm{ P u}_{L^2}^2+\norm{u}_{L^2}^2}.
  \end{equation*}
  Consequently,  from \eqref{est1+++},  \eqref{pumu}  and the relationship  
\begin{eqnarray*}
	 {\rm Re} \comi{X_0u,~\mathcal M   u}_{L^2}= {\rm Re} \comi{P u,~\mathcal M   u}_{L^2}- {\rm Re}\sum_{j=1}^n\comi{X_j^*X_j u,~\mathcal M   u}_{L^2}, 
\end{eqnarray*} 
it follows that 
  \begin{eqnarray*}
  	&& \Big<\frac{ (1+\abs y^8) \abs{\partial_x
  V(x)}^2 }{\rho ^3} \frac{y^2}{1+y^2}u,~u\Big>_{L^2}+ \Big<\frac{2\rho }{(1+y^2)^2}u,~u\Big>_{L^2}  \\
  &&\qquad\qquad- \Big<\frac{(1+\abs y^8)  \sum_{1\leq i,j\leq n}\inner{\partial_{x_ix_j}V} y_iy_j}{\rho ^3}\frac{y^2}{1+y^2}u,~u\Big>_{L^2}\\
  &\leq& \eps  \Big<\frac{(1+\abs y^8) \abs{\partial_{x}
  V(x)}^2}{\rho ^3} \frac{y^2}{1+y^2}u,~u\Big>_{L^2}+C_\eps \inner{ \norm{ P u}_{L^2}^2+\norm{u}_{L^2}^2}.   
  \end{eqnarray*}
  As the result, for all $\sigma$ with $0<\sigma<1,$ letting $\eps=1-\sigma$  gives the desired estimate in Lemma \ref{prpkey}.  The proof is thus complete. 
  \end{proof}

The rest  part is devoted to the proof of  Theorem \ref{mtm+}.

  \begin{proof}[Proof of  Theorem \ref{mtm+}]
For the symmetric Hessian matrix $\inner{\partial_{x_ix_j} V}_{1\leq i, j\leq  n}$, we can find  a $n\times n$   orthogonal matrix $Q(x)=\big(q_{ij}(x)\big)_{1\leq i,j\leq n}$ such that  
\begin{eqnarray}\label{diag}
	Q^T\begin{pmatrix}
		\lambda_1\\
		&\lambda_2\\
		&&\ddots\\
		&&&\lambda_n
	\end{pmatrix} Q=\inner{\partial_{x_ix_j}V}_{1\leq i,j\leq n},
\end{eqnarray}
where $\lambda_j, 1\leq j\leq n$ are the eigenvalues of  the Hessian $\inner{\partial_{x_ix_j} V}_{1\leq i, j\leq  n}.$
Then  for any $x\in \mathbb R^n$ we  can write 
  \begin{eqnarray}\label{matrix}
  -\sum_{j\notin I_x}\lambda_j (x)\com{ \inner{Q(x)y}_j}^2=	-\sum_{1\leq i,j\leq n}\inner{\partial_{x_ix_j}V(x)} y_iy_j+\sum_{j\in I_x}\lambda_j(x) \com{ \inner{Q(x)y}_j}^2,
  \end{eqnarray}
  where  $\inner{Q(x)y}_j$ stands for the $j$-th component of  the vector $Q(x)y$, and      
  \begin{eqnarray*}
  	I_x= \Big\{1\leq \ell \leq n;~\lambda_\ell(x)>0\Big\}. 
  \end{eqnarray*}
  Thus it follows from \eqref{matrix}  and the assumption \eqref{assext} that ,  for any $x\in\mathbb R^n$, 
   \begin{eqnarray*}
  \sum_{j\notin I_x}\inner{-\lambda_j (x)} \com{ \inner{Q(x)y}_j}^2\leq -\sum_{1\leq i,j\leq n}\inner{\partial_{x_ix_j}V(x)} y_iy_j   + C \comi{\partial_x
  V(x)}^{4/3} \abs{y}^2.
  \end{eqnarray*}
  This together with the estimate in  Lemma \ref{prpkey} yields, for all $0<\sigma<1$ and  for  any  $\eps>0, $
 \begin{eqnarray*}
  	&& \sigma \Big<\frac{ (1+\abs y^8) \comi{\partial_x
  V(x)}^2}{\rho ^3} \frac{y^2}{1+y^2}u,~u\Big>_{L^2}+ \Big<\frac{\rho }{\comi y^4}u,~u\Big>_{L^2}  \\
  &&\qquad\qquad + \int_{\mathbb R^n}   \bigg(\sum_{j\notin I_x} \int_{\mathbb R^n}  \frac{(1+\abs y^8) \inner{-\lambda_j (x)}\com{ \inner{Q(x)y}_j}^2    }{\rho ^3}\frac{y^2}{1+y^2}u^2\, dy\bigg) \,dx\\
  &\leq&  C \Big<\frac{(1+\abs y^8)   \comi{\partial_x
  V(x)}^{4/3} \abs{y}^2 }{\rho ^3}\frac{y^2}{1+y^2}u,~u\Big>_{L^2}+ C_\sigma \inner{ \norm{ P u}_{L^2}^2+\norm{u}_{L^2}^2}\\
  &\leq&  \eps  \Big<\frac{ (1+\abs y^8) \comi{\partial_x
  V(x)}^{2}    }{\rho ^3}\frac{y^2}{1+y^2}u,~u\Big>_{L^2}+C_\eps \Big<\frac{(1+\abs y^8)   \comi y^6 }{\rho ^3}\frac{y^2}{1+y^2}u,~u\Big>_{L^2}\\
  &&+ C_\sigma \inner{ \norm{ P u}_{L^2}^2+\norm{u}_{L^2}^2}\\
  &\leq&  \eps  \Big<\frac{(1+\abs y^8)   \comi{\partial_x
  V(x)}^{2}  }{\rho ^3}\frac{y^2}{1+y^2}u,~u\Big>_{L^2}+  C _{\eps,\sigma} \inner{ \norm{ P u}_{L^2}^2+\norm{u}_{L^2}^2}, 
  \end{eqnarray*}
  the second and last inequalities  holding because  
  \begin{eqnarray*}
  	\comi{\partial_x
  V(x)}^{4/3} \abs{y}^2\leq \eps  \comi{\partial_x
  V(x)}^{2} +C_\eps \abs y^{6} 
  \end{eqnarray*}
  and 
  \begin{eqnarray*}
  	\Big<\frac{(1+\abs y^8)   \comi y^6 }{\rho ^3}\frac{y^2}{1+y^2}u,~u\Big>_{L^2}\leq C\norm{\comi y u}_{L^2}^2\leq C \inner{ \norm{ P u}_{L^2}^2+\norm{u}_{L^2}^2}.
  \end{eqnarray*}
  Now letting $\eps=\frac{\sigma}{2}$, we obtain 
  \begin{eqnarray*}
  		&& \frac{\sigma}{2}\Big<\frac{  (1+\abs y^8)\comi{\partial_x
  V(x)}^2 }{\rho ^3} \frac{y^2}{1+y^2}u,~u\Big>_{L^2}+ \Big<\frac{\rho }{\comi y^4}u,~u\Big>_{L^2}  \\
   &&\qquad\qquad + \int_{\mathbb R^n}  \bigg(\sum_{j\notin I_x} \int_{\mathbb R^n}  \frac{(1+\abs y^8) \inner{-\lambda_j (x)}\com{ \inner{Q(x)y}_j}^2    }{\rho ^3}\frac{y^2}{1+y^2}u^2\, dy\bigg) \,dx\\
  &\leq&   C_\sigma \inner{ \norm{ P u}_{L^2}^2+\norm{u}_{L^2}^2},
  \end{eqnarray*}
  and thus, choosing $\sigma=1/2,$
    \begin{eqnarray}\label{uplo}
  		&& \Big<\frac{ (1+\abs y^8) \comi{\partial_x
  V(x)}^2}{\rho ^3} \frac{y^2}{1+y^2}u,~u\Big>_{L^2}+ \Big<\frac{\rho }{\comi y^4}u,~u\Big>_{L^2}\nonumber  \\
   &&\qquad\qquad + \int_{\mathbb R^n}    \bigg(\sum_{j\notin I_x} \int_{\mathbb R^n}  \frac{  \inner{-\lambda_j (x)}\com{ \inner{Q(x)y}_j}^2  y^2    }{\comi{\partial_x V(x)}^{3} \comi y^{6}} u^2\, dy\bigg)\,dx \\
  &\leq&   C \inner{ \norm{ P u}_{L^2}^2+\norm{u}_{L^2}^2}\nonumber
  \end{eqnarray}
  due to the fact that  $-\lambda_j(x)\geq 0$ for $j\notin I_x$ and 
      \begin{eqnarray*}
  	\frac{1}{\comi{\partial_x V(x)}^{3} \comi y^{6}}\leq  C  \frac{(1+\abs y^8)}{\rho ^3}\frac{1}{1+y^2}.
  \end{eqnarray*}
  In the following discussions we will give  the lower bound of the summation on the left side of \eqref{uplo}.  To do so, we use the 
 the estimates
  \begin{eqnarray*}
	\comi{\partial_x V(x)}^{1\over 4}  \frac{y^2}{1+y^2} &\leq& \inner{\frac{1}{8} \frac{ \comi{\partial_x
  V(x)}^2\comi y^8 }{\rho ^3}  + \frac{3}{8}   \frac{\rho}{\comi y^4}+\frac{\comi y}{2}}  \frac{y^2}{1+y^2} \\
  &\leq&   C \frac{ \comi{\partial_x
  V(x)}^2 (1+\abs y^8)}{ \rho^3}  \frac{y^2}{1+y^2}   +  \frac{\rho}{\comi y^4} +\comi y,\end{eqnarray*}
together with \eqref{uplo},   to conclude 
  \begin{eqnarray}\label{1over8}
  	   \sum_{j=1}^n \norm{\comi{\partial_x V(x)}^{1\over 8} y_j \comi{y}^{-1}  u}_{L^2}^2  \leq   C \inner{ \norm{ P u}_{L^2}^2+\norm{\comi y^{1/2}u}_{L^2}^2}\leq   C \inner{ \norm{ P u}_{L^2}^2+\norm{ u}_{L^2}^2}.
  \end{eqnarray}
  Moreover applying  \eqref{comest} with $v=\comi{y}^{-1}  u$ and $\theta(x)\cdot y=\comi{\partial_x V(x)}^{1\over 8}y_j,$ we get 
\begin{eqnarray*}
		  \norm{\comi{\partial_x V(x)}^{1\over 16}  \comi{y}^{-1}  u}_{L^2}^2 &\leq&  C  \inner{\norm{\comi{\partial_x V(x)}^{1\over 8} y_j \comi{y}^{-1}  u}_{L^2}^2+ \norm{\partial_{ y_j} \comi{y}^{-1}  u}_{L^2}^2}\\
		  &\leq&  C  \inner{\norm{\comi{\partial_x V(x)}^{1\over 8} y_j \comi{y}^{-1}  u}_{L^2}^2+ \norm{\partial_{ y_j}  u}_{L^2}^2+\norm{ u}_{L^2}^2}\\ 
		  &\leq&  C  \inner{  \norm{P u}_{L^2}^2+\norm{ u}_{L^2}^2},
		  \end{eqnarray*}
		  the last inequality following from \eqref{1over8} and \eqref{revel}.    As a result, observe
		  \begin{eqnarray*}
		  	 &&\norm{\comi{\partial_x V(x)}^{1\over 16}  u}_{L^2}^2\\
		  	 &=&\inner{\comi{\partial_x V(x)}^{1\over 16} \frac{1}{1+\abs{y}^2} u,  ~\comi{\partial_x V(x)}^{1\over 16}  u}_{L^2}+\sum_{j=1}^n\inner{\comi{\partial_x V(x)}^{1\over 16} \frac{y_j^2}{1+\abs{y}^2} u,  ~\comi{\partial_x V(x)}^{1\over 16}  u}_{L^2}\\
		  	 &=& \norm{\comi{\partial_x V(x)}^{1\over 16}  \comi{y}^{-1}  u}_{L^2}^2+\sum_{j=1}^n\norm{\comi{\partial_x V(x)}^{1\over 16} y_j \comi{y}^{-1}  u}_{L^2}^2 \\
		  	 &\leq& \norm{\comi{\partial_x V(x)}^{1\over 16}  \comi{y}^{-1}  u}_{L^2}^2+\sum_{j=1}^n\norm{\comi{\partial_x V(x)}^{1\over 8} y_j \comi{y}^{-1}  u}_{L^2}^2,
		  \end{eqnarray*}
		  and thus combining the above inequalities and \eqref{1over8},  we obtain
\begin{eqnarray}\label{estimate+}
	 \norm{\comi{\partial_x V(x)}^{1\over 16}  u}_{L^2}^2 		   \leq   C  \inner{  \norm{P u}_{L^2}^2+\norm{ u}_{L^2}^2}.
\end{eqnarray}
Then the conclusion (i) in Theorem \ref{mtm+} follows.  

Now we prove the conclusion (ii).  Let  $x\in\mathbb R^n $ be given and let  $1\leq i, \ell \leq n$ and   $j\notin I_x.$   Recall $Q(x)=\big(q_{k\ell}(x)\big)_{1\leq k,\ell\leq n}.$    
Similarly as above,  applying again  \eqref{comest} with 
\begin{eqnarray*}
	v= y_\ell \comi{y}^{-3}   u,\quad \theta(x)\cdot y=\inner{-\lambda_j (x)}^{1/2} \comi{\partial_xV(x)}^{-{3\over 2}} \inner{Q(x)y}_j=\inner{-\lambda_j (x)}^{1/2} \comi{\partial_xV(x)}^{-{3\over 2}} \sum_{k=1}^n q_{jk}y_k,
\end{eqnarray*}
we have,  
\begin{eqnarray*}
		  &&  \norm{\inner{-\lambda_j (x)}^{1/4} \comi{\partial_xV(x)}^{-3/4}\abs{q_{ji}(x)}^{1/2}y_\ell  \comi{y}^{-3}    u}_{L^2(\mathbb R_y^n)}^2 \\
		  &\leq&  C    \inner{\norm{\inner{-\lambda_j (x)}^{1/2} \comi{\partial_xV(x)}^{-3/2} \com{\inner{Q(x)y}_j } y_\ell \comi{y}^{-3}   u}_{L^2(\mathbb R_y^n)}^2+ \norm{ \partial_{y_i} y_\ell \comi{y}^{-3}  u}_{L^2(\mathbb R_y^n)}^2}\\
		   &\leq&   C  \int_{\mathbb R^n}  \frac{   \inner{-\lambda_j (x)}\com{ \inner{Q(x)y}_j}^2  y^2  }{\comi{\partial_x V(x)}^{3} \comi y^{6}} u^2\, dy+C \inner{ \norm{ \partial_{y_i}     u}_{L^2(\mathbb R_y^n)}^2+\norm{u}_{L^2(\mathbb R_y^n)}^2}.
		  \end{eqnarray*}
Thus, combining  \eqref{uplo} and \eqref{revel}, 
\begin{eqnarray}\label{eiglow}
		  && \int_{\mathbb R^n }    \inner{\sum_{i=1}^n\sum_{j\notin I_x} \norm{\inner{-\lambda_j (x)}^{1/4} \comi{\partial_xV(x)}^{-3/4}\abs{q_{ji}(x)}^{1/2} y_\ell  \comi{y}^{-3}   u}_{L^2(\mathbb R_y^n)}^2}\,dx\\
		  &\leq&  C \int_{\mathbb R^n}    \bigg(\sum_{j\notin I_x} \int_{\mathbb R^n}  \frac{   \inner{-\lambda_j (x)}\com{ \inner{Q(x)y}_j}^2  y^2  }{\comi{\partial_x V(x)}^{3} \comi y^{6}} u^2\, dy\bigg)\,dx +C  \sum_{i=1}^n \inner{  \norm{ \partial_{y_i}     u}_{L^2}^2+\norm{u}_{L^2}^2}\nonumber \\
  &\leq&   C \inner{ \norm{ P u}_{L^2}^2+\norm{u}_{L^2}^2}.\nonumber
\end{eqnarray}
Moreover, using again   \eqref{comest} with 
\begin{eqnarray*}
	v= \comi{y}^{-3}   u,\quad \theta\cdot y=\inner{-\lambda_j (x)}^{1/4} \comi{\partial_xV(x)}^{-3/4}  \abs{q_{ji}(x)}^{1/2} y_\ell, 
\end{eqnarray*}
gives, 
\begin{eqnarray}\label{+eiglow}
		 &&\int_{\mathbb R^n }    \inner{\sum_{i=1}^n\sum_{j\notin I_x} \norm{  \inner{-\lambda_j (x)}^{1/8}\comi{\partial_xV(x)}^{-3/8}\abs{q_{ji}(x)}^{1/4}   \comi{y}^{-3}  u}_{L^2(\mathbb R_y^n)} }\,dx\nonumber \\
		  &\leq &  \int_{\mathbb R^n }    \inner{\sum_{i=1}^n\sum_{j\notin I_x} \norm{\inner{-\lambda_j (x)}^{1/4} \comi{\partial_xV(x)}^{-3/4}\abs{q_{ji}(x)}^{1/2} y_\ell  \comi{y}^{-3}   u}_{L^2(\mathbb R_y^n)}^2}\,dx+C     \norm{ \partial_{y_\ell }  \comi{y}^{-3}  u}_{L^2}^2 \nonumber\\
  &\leq&   C \inner{ \norm{ P u}_{L^2}^2+\norm{u}_{L^2}^2}.
\end{eqnarray}
the last inequality following from \eqref{eiglow} and \eqref{revel}.  On the other hand,  in view of \eqref{diag}    we see 
\begin{eqnarray}\label{trace}
	 \sum_{i=1}^n\sum_{j=1}^n\lambda_j(x) \inner{q_{ji}(x)}^2=\sum_{j=1}^n\lambda_j(x)= \Delta_x V(x). 
\end{eqnarray}
Then,   by the assumption \eqref{longtimebe} in  Theorem \ref{mtm+},    we can find    a constant $C_\alpha$ depending on $\alpha$, such that 
 \begin{eqnarray*}
	\forall~x\in \mathbb R^n,\quad 0\leq  \alpha \abs{ \partial_x V(x)}^2-\Delta_x V(x)+C_\alpha \leq C_\alpha \comi{ \partial_x V(x)}^2+\sum_{i=1}^n\sum_{j\notin I_x}^n \inner{-\lambda_j(x)} \abs{q_{ji}(x)}^2, 
\end{eqnarray*}
the last inequality  following from \eqref{trace}.   And  thus for any $x\in\mathbb R^n,$
\begin{eqnarray*}
	\abs{ \alpha \abs{ \partial_x V(x)}^2-\Delta_x V(x)}^{1/4} \leq C    \comi{ \partial_x V(x)}^{1/2}+C \sum_{i=1}^n\sum_{j\notin I_x}^n \inner{-\lambda_j(x)}^{1/4} \abs{q_{ji}(x)}^{1/2},
\end{eqnarray*}
which, together with \eqref{+eiglow},     yields
 \begin{eqnarray*}
 	 &&  \Big<\frac{\big|\alpha \abs{ \partial_x V(x)}^2-\Delta_x V(x)\big|^{1\over 4}}{ \comi{\partial_xV(x)}^{3/4}\comi{y}^{6} }  u,~u \Big>_{L^2}  \\
 	 &\leq & C \Big<\frac{  \comi{ \partial_x V(x)}^{1/2}}{ \comi{\partial_xV(x)}^{3/4}\comi{y}^{6} }  u,~u \Big>_{L^2}+C\int_{\mathbb R^n }    \inner{\sum_{i=1}^n\sum_{j\notin I_x} \Big<\frac{\inner{-\lambda_j (x)}^{1/4}\abs{q_{ji}(x)}^{1/2}}{ \comi{\partial_xV(x)}^{3/4}\comi{y}^{6} }  u,~u \Big>_{L^2(\mathbb R_y^n)} }\,dx\\
		  &\leq&  C \inner{ \norm{ P u}_{L^2}^2+\norm{u}_{L^2}^2}.
 \end{eqnarray*}
 As a result,  we conclude,  combining \eqref{estimate+} , \eqref{revel} and the above inequality,  
 \begin{eqnarray*}
 	 \norm{|\alpha \abs{ \partial_x V(x)}^2-\Delta_x V(x)\big|^{1/ 80}  u }_{L^2}^2 
		  \leq C \inner{ \norm{ P u}_{L^2}^2+\norm{u}_{L^2}^2},
 \end{eqnarray*}
 due to the estimate
  \begin{eqnarray*}
 	|\alpha \abs{ \partial_x V(x)}^2-\Delta_x V(x)\big|^{1/ 40}\leq \frac{1}{10} \frac{\big|\alpha \abs{ \partial_x V(x)}^2-\Delta_x V(x)\big|^{1\over 4}}{ \comi{\partial_xV(x)}^{3/4}\comi{y}^{6} }+\frac{3}{5}  \comi{\partial_xV(x)}^{1/8}+\frac{3}{10}  \comi{y}^{2}. 
 \end{eqnarray*}
 Thus the proof  of Theorem \ref{mtm+} is complete.  
    \end{proof}

\bigskip

\subsection{Proof of Theorem \ref{+Hypo}} \label{sub21} This subsection is devoted to proving Theorem \ref{+Hypo}.  Similarly as Lemma \ref{prpkey} we have the following 

\begin{lemma}\label{prpkey+}
Let $V(x)\in C^2\inner{\mathbb R^n}$. Then for all $\sigma \in\,]0,1[$ there exists a constant $C_\sigma>0$ such that for any $u\in C_0^\infty(\mathbb R^{2n})$  we have
 	\begin{eqnarray*}
 	\begin{aligned}
  	& \sigma\Big<\frac{ \comi{\partial_x
  V(x)}^2  }{\comi{\partial_xV(x)\cdot y }^3} \frac{y^2}{1+y^2}u,~u\Big>_{L^2}+ \Big<\frac{2\comi{\partial_xV(x)\cdot y }}{\comi y^4}u,~u\Big>_{L^2}  - \sum_{1\leq i,j\leq n}\Big<\frac{ \inner{\partial_{x_ix_j}V} y_iy_j}{\comi{\partial_xV(x)\cdot y }^3}\frac{y^2}{1+y^2}u,~u\Big>_{L^2}\\
  &\leq  C_\sigma \inner{ \norm{ P u}_{L^2}^2+\norm{u}_{L^2}^2}.
  \end{aligned}
  \end{eqnarray*}
\end{lemma}

\begin{proof} The proof is quite similar as  Lemma \ref{prpkey}. 
  Let  $\mathcal K \in C^1(\mathbb R^{2n})$  be  defined  by
\begin{eqnarray*}
	\mathcal K=\mathcal K (x,y)= \frac{2 \partial_xV(x)\cdot y }{\comi{\partial_xV(x)\cdot y} } \frac{y^2}{1+y^2}.
\end{eqnarray*} 	
Then  using the relation
\begin{eqnarray*}
	\frac{1}{2}\big[\mathcal K ,~X_0\big]=\frac{\partial_xV(x)\cdot y }{\comi{\partial_xV(x)\cdot y}  }\Big[ \frac{y^2}{(1+y^2)},~X_0\Big]+ \Big[ \frac{\partial_xV(x)\cdot y }{\comi{\partial_xV(x)\cdot y} },~X_0 \Big]\frac{y^2}{(1+y^2)},
\end{eqnarray*}
we obtain
\begin{eqnarray*}
	\frac{1}{2}\big[\mathcal K ,~X_0\big]
	&=&\frac{2\abs{\partial_xV(x)\cdot y}^2}{\comi{\partial_xV(x)\cdot y} (1+y^2)^2}+  \inner{\frac{ \abs{\partial_x
  V(x)}^2 }{\comi{\partial_xV(x)\cdot y} ^3} -  \sum_{1\leq i,j\leq n}\frac{\inner{\partial_{x_ix_j}V} y_iy_j}{\comi{\partial_xV(x)\cdot y} ^3}}\frac{y^2}{1+y^2}.
\end{eqnarray*}
Thus using the relationship 
\begin{eqnarray*}
	{\rm Re}\comi{X_0 u,~\mathcal K u}_{L^2}=
  \frac{1}{2}\comi{\com{\mathcal K ,~X_0}u,~u}_{L^2},
\end{eqnarray*}
 we conclude
    \begin{eqnarray}\label{est23}
 && \Big<\frac{ \comi{\partial_x
  V(x)}^2 }{\comi{\partial_xV(x)\cdot y}  ^3} \frac{y^2}{1+y^2}u,~u\Big>_{L^2}+ \Big<\frac{2\comi{\partial_xV(x)\cdot y}}{ (1+y^2)^2}u,~u\Big>_{L^2}\nonumber \\
  &&\qquad\qquad- \sum_{1\leq i,j\leq n}\Big<\frac{   \inner{\partial_{x_ix_j}V} y_iy_j}{\comi{\partial_xV(x)\cdot y}^3}\frac{y^2}{1+y^2}u,~u\Big>_{L^2}\\
  &\leq & {\rm Re}\comi{X_0 u,~\mathcal K u}_{L^2}+3\norm{u}_{L^2}^2.\nonumber
   \end{eqnarray}
On the other hand,  since $\mathcal K \in L^\infty(\mathbb R^{2n})$ with $\norm{\mathcal K }_{L^\infty}\leq 2$ then
 it is easy to see 
 \begin{equation}\label{+++pumu+}
 	\abs{{\rm Re}\comi{P u,~\mathcal K u}_{L^2}}\leq 
  \norm{Pu}_{L^2}^2+ \norm{u}_{L^2}^2.
 \end{equation} 
Recall $X_j=\partial_{y_j}+\frac{y_j}{2},j=1,\cdots n.$ Then direct computation gives,    for each $1\leq j\leq n$, 
\begin{eqnarray*}
 \com{\mathcal K ,~X_j}&=&-\frac{2\partial_{x_j} V(x) }{\comi{\partial_xV(x)\cdot y }^3 }\frac{y^2}{1+y^2}-\frac{4\inner{\partial_{x} V(x) \cdot y }y_j}{\comi{\partial_xV(x)\cdot y } (1+y^2)^2}.
 \end{eqnarray*} 
As a result,  for any $\eps>0,$  
\begin{eqnarray*}
	&&\abs{ {\rm Re} \comi{ X_j u,~\com{\mathcal K ,~ X_j} u}_{L^2}	} \\
	&\leq &    \eps \norm{\comi{\partial_xV(x)\cdot y }^{-3}  (1+y^2)^{-1}y^2\abs{\partial_{x_j} V(x) } u}_{L^2}^2	  +C_ \eps \inner{ \norm{ X_j u}_{L^2}^2+\norm{u}_{L^2}^2}\\	
	&\leq &\eps  \Big<\frac{ \abs{\partial_{x_j}
  V(x)}^2 }{\comi{\partial_xV(x)\cdot y }^3} \frac{y^2}{1+y^2}u,~u\Big>_{L^2}+C_ \eps \inner{ \norm{ X_j u}_{L^2}^2+\norm{u}_{L^2}^2}.
\end{eqnarray*}
This yields, using again the facts that $\norm{\mathcal K }_{L^\infty}\leq 2$ and \eqref{revel},  
\begin{eqnarray*}
\sum_{j=1}^n \abs{ {\rm Re} \comi{X_j^*X_j u,~\mathcal K   u}_{L^2}	} &\leq & \sum_{j=1}^n\inner{\abs{{\rm Re} \comi{ X_j u,~\mathcal K  X_j u}_{L^2}	}+ \abs{ {\rm Re} \comi{ X_j u,~\com{\mathcal K ,~ X_j} u}_{L^2}	}}\\
&\leq& \eps  \Big<\frac{ \abs{\partial_{x}
  V(x)}^2 }{\comi{\partial_xV(x)\cdot y }^3} \frac{y^2}{1+y^2}u,~u\Big>_{L^2}+C_ \eps \inner{ \norm{ P u}_{L^2}^2+\norm{u}_{L^2}^2}, 
\end{eqnarray*}
which,   along with  \eqref{est23},  \eqref{+++pumu+}  and the relationship that 
\begin{eqnarray*}
	 {\rm Re} \comi{X_0u,~\mathcal K   u}_{L^2}= {\rm Re} \comi{P u,~\mathcal K   u}_{L^2}- {\rm Re}\sum_{j=1}^n\comi{X_j^*X_j u,~\mathcal K   u}_{L^2}, 
\end{eqnarray*} 
implies 
  \begin{eqnarray*}
  	&& \Big<\frac{ \comi{\partial_x
  V(x)}^2 }{\comi{\partial_xV(x)\cdot y }^3} \frac{y^2}{1+y^2}u,~u\Big>_{L^2}+ \Big<\frac{2\comi{\partial_xV(x)\cdot y } }{(1+y^2)^2}u,~u\Big>_{L^2}  \\
  &&\qquad\qquad-\sum_{1\leq i,j\leq n}\Big<\frac{ \inner{\partial_{x_ix_j}V} y_iy_j}{\comi{\partial_xV(x)\cdot y }^3}\frac{y^2}{1+y^2}u,~u\Big>_{L^2}\\
  &\leq& \eps  \Big<\frac{ \abs{\partial_{x}
  V(x)}^2}{\comi{\partial_xV(x)\cdot y }^3} \frac{y^2}{1+y^2}u,~u\Big>_{L^2}+C_\eps \inner{ \norm{ P u}_{L^2}^2+\norm{u}_{L^2}^2}.   
  \end{eqnarray*}
 Given any $\sigma\in\,]0,1[,$ letting $\eps=1-\sigma$ gives the desired estimate in Lemma \ref{prpkey+}. The proof is thus complete. 
  \end{proof}
  
   \begin{lemma} \label{lemm+2}
  Let $\tau\geq 0$ be given.   Then for any $\eps>0$, 
  \begin{eqnarray*}
  && \sum_{1\leq i,j\leq n} \tau\,\Big<\frac{  \comi{\partial_xV(x)}^{4/5}\inner{\partial_{x_i}V(x)} \inner{\partial_{x_j}V(x)} y_iy_j}{\comi{\partial_x V(x)\cdot y}^3}\frac{y^2}{1+y^2}u,~u\Big>_{L^2}\\
   &\leq&  \eps  \,\Big<\frac{  \comi{\partial_x
  V(x)}^{2}  }{\comi{\partial_x V(x)\cdot y}^3}\frac{y^2}{1+y^2}u,~u\Big>_{L^2}+\eps  \,\Big<\frac{  \comi{\partial_x V(x)\cdot y} }{\comi{y} ^4}u,~u\Big>_{L^2}+  C _{\eps,\tau} \inner{ \norm{ P u}_{L^2}^2+\norm{u}_{L^2}^2},
   \end{eqnarray*}
   where $C _{\eps,\tau}$ is a constant depending only on $\eps$ and $\tau$.  \end{lemma}

\begin{proof} 
In the proof we use $C_{\eps,\tau}$ to denote the different constants depending on $\eps$ and $\tau$.   
Direct  calculation gives  
    \begin{eqnarray*}
  	\tau \frac{\comi{\partial_xV}^{4/5}\abs{\partial_x V(x)\cdot y}^2 }{\comi{\partial_x V(x)\cdot y}^3} \frac{\abs y^2}{1+\abs y^2} & \leq &  \frac{\eps \comi{\partial_xV(x)}^{2}+C_{\eps,\tau} \abs{\partial_x V(x)\cdot y}^{10/3
  	}}{\comi{\partial_x V(x)\cdot y}^3} \frac{\abs y^2}{1+\abs y^2} \\
  	& \leq &  \eps \,   \frac{\comi{\partial_xV(x)}^{2}}{\comi{\partial_x V(x)\cdot y}^3} \frac{\abs y^2}{1+\abs y^2}+  C_{\eps,\tau} \comi{\partial_x V(x)\cdot y}^{1/3
  	}\\
  	& \leq &    \eps \, \frac{\comi{\partial_xV(x)}^{2}}{\comi{\partial_x V(x)\cdot y}^3} \frac{\abs y^2}{1+\abs y^2}+  \eps\,\frac{ \comi{\partial_x V(x)\cdot y}}{\comi y^4}+ C_{\eps,\tau}\comi y^2,
  	  \end{eqnarray*}
and thus,  using \eqref{revel},
\begin{eqnarray*}
&&	\tau\,\Big<\frac{    \comi{\partial_xV(x)}^{4/5}\abs{\partial_x V(x)\cdot y}^2}{\comi{\partial_x V(x)\cdot y}^3}\frac{y^2}{1+y^2}u,~u\Big>_{L^2} \\
   &\leq&  \eps  \,\Big<\frac{  \comi{\partial_x
  V(x)}^{2}  }{ \comi{\partial_x V(x)\cdot y}^3}\frac{y^2}{1+y^2}u,~u\Big>_{L^2}+\eps  \,\Big<\frac{  \comi{\partial_x V(x)\cdot y} }{\comi{y} ^4}u,~u\Big>_{L^2}+  C _{\eps,\tau} \inner{ \norm{ P u}_{L^2}^2+\norm{u}_{L^2}^2}.
\end{eqnarray*}
 Then   observing 
\begin{eqnarray*}
  	     \sum_{1\leq i,j\leq n} \tau\,\Big<\frac{   \comi{\partial_xV}^{4/5}\inner{\partial_{x_i}V} \inner{\partial_{x_j}V} y_iy_j}{\rho ^3}\frac{y^2}{1+y^2}u,~u\Big>_{L^2}
  	     = \tau\,\Big<\frac{\comi{\partial_xV}^{4/5}\abs{\partial_x V\cdot y}^2}{\rho ^3}\frac{y^2}{1+y^2}u,~u\Big>_{L^2},
  \end{eqnarray*}
 the desired estimate in Lemma \ref{lemm+2} follows.  The proof is   complete. 
\end{proof}

The rest is occupied by  the proof   of Theorem \ref{+Hypo}. 

\begin{proof}[Proof of Theorem \ref{+Hypo}]
	  By virtue of Lemma \ref{prpkey+} and Lemma \ref{lemm+2},  we obtain, for all $0<\sigma<1$ and for any $\eps>0,$
	  \begin{eqnarray*}
	  \begin{aligned}
  	& \sigma \Big<\frac{ \comi{\partial_x
  V(x)}^2  }{\comi{\partial_xV(x)\cdot y }^3} \frac{y^2}{1+y^2}u,~u\Big>_{L^2}+ \Big<\frac{2\comi{\partial_xV(x)\cdot y }}{\comi y^4}u,~u\Big>_{L^2}\\
  &\quad+\sum_{1\leq i,j\leq n} \set{\tau\,\Big<\frac{  \comi{\partial_xV}^{4/5}\inner{\partial_{x_i}V} \inner{\partial_{x_j}V} y_iy_j}{\comi{\partial_x V\cdot y}^3}\frac{y^2}{1+y^2}u,~u\Big>_{L^2}  -  \Big<\frac{ \inner{\partial_{x_ix_j}V} y_iy_j}{\comi{\partial_xV(x)\cdot y }^3}\frac{y^2}{1+y^2}u,~u\Big>_{L^2}}\\
  &\leq   \eps  \,\Big<\frac{  \comi{\partial_x
  V(x)}^{2}  }{ \comi{\partial_x V(x)\cdot y}^3}\frac{y^2}{1+y^2}u,~u\Big>_{L^2}+\eps  \,\Big<\frac{  \comi{\partial_x V(x)\cdot y} }{\comi{y} ^4}u,~u\Big>_{L^2}+  C _{\eps,\tau,\sigma} \inner{ \norm{ P u}_{L^2}^2+\norm{u}_{L^2}^2}.
  \end{aligned}
  \end{eqnarray*}
  Letting $\eps=\sigma/2$,    denoting by $C$ the different constants which may depend on $\tau$ and $\sigma,$  we have
   \begin{eqnarray*}
   \begin{aligned}
  	&\frac{\sigma}{2} \Big<\frac{ \comi{\partial_x
  V(x)}^2  }{\comi{\partial_xV(x)\cdot y }^3} \frac{y^2}{1+y^2}u,~u\Big>_{L^2}+ \Big<\frac{\comi{\partial_xV(x)\cdot y }}{\comi y^4}u,~u\Big>_{L^2}\\
  &\quad+\sum_{1\leq i,j\leq n} \set{\tau\,\Big<\frac{  \comi{\partial_xV}^{4/5}\inner{\partial_{x_i}V} \inner{\partial_{x_j}V} y_iy_j}{\comi{\partial_x V\cdot y}^3}\frac{y^2}{1+y^2}u,~u\Big>_{L^2}  -  \Big<\frac{ \inner{\partial_{x_ix_j}V} y_iy_j}{\comi{\partial_xV(x)\cdot y }^3}\frac{y^2}{1+y^2}u,~u\Big>_{L^2}}\\
  &\leq    C \inner{ \norm{ P u}_{L^2}^2+\norm{u}_{L^2}^2},
  \end{aligned}
  \end{eqnarray*}
  and thus,  using \eqref{revel},  
   \begin{eqnarray}\label{mares}
   \begin{aligned}
  	&\frac{\sigma}{2}  \Big<\frac{ \comi{\partial_x
  V(x)}^2  }{\comi{\partial_xV(x)\cdot y }^3} \frac{y^2}{1+y^2}u,~u\Big>_{L^2}+ \Big<\frac{\comi{\partial_xV(x)\cdot y }}{\comi y^4}u,~u\Big>_{L^2} + \Big<\frac{  y^TA_\tau(x) y}{\comi{\partial_x V\cdot y}^3}\frac{y^2}{1+y^2}u,~u\Big>_{L^2}   \\
  &\leq     C \inner{ \norm{ P u}_{L^2}^2+\norm{u}_{L^2}^2},
  \end{aligned}
  \end{eqnarray}
where  $A_\tau(x)$ is the  matrix  defined in  Theorem \ref{+Hypo}, i.e.,  \begin{eqnarray*}
A_\tau(x)=\inner{a_{ij}^{\tau}(x)}_{1\leq i, j\leq n},\qquad a_{ij}^{\tau}(x)=\tau\comi{\partial_x V(x)}^{4\over 5} \inner{\partial_{x_i}V(x)}\inner{\partial_{x_j}V(x)}- \partial_{x_ix_j}V(x)+\tau \delta_{ij}.
 \end{eqnarray*}
Now under the assumption that  $A_\tau(x)$ is positive-definite, we can find  its Cholesky decomposition matrix $$B_\tau(x)=\inner{b_{ij}^\tau(x)}_{1\leq i, j\leq n},$$
 satisfying the relation
   \begin{eqnarray*}
  	  A_\tau=B_\tau^TB_\tau. 
  \end{eqnarray*}
 Then using  the following estimates
 \begin{eqnarray*}
 	  &&\sum_{\ell=1}^n\sum_{j=1}^n\Big\|\sum_{1\leq k\leq n} b_{jk}^\tau(x) \comi{\partial_x V(x)}^{-3/2} y_k  y_\ell  \comi{y}^{-5/2} u \Big\|_{L^2}^2\\
 	 &=&  \Big<\frac{   B_\tau(x) y}{\comi{\partial_x V(x)}^3\comi y^3}\frac{y^2}{1+y^2}u,~  B_\tau(x) y \,u\Big>_{L^2} \\
 	 &\leq & C  \Big<\frac{   B_\tau(x) y}{\comi{\partial_x V(x) \cdot y}^3}\frac{y^2}{1+y^2}u,~  B_\tau(x) y \,u\Big>_{L^2} 
 	 =  C  \Big<\frac{   y^TA_\tau(x) y}{\comi{\partial_x V(x) \cdot y}^3}\frac{y^2}{1+y^2}u,~  u\Big>_{L^2} 
 \end{eqnarray*}
 and \eqref{mares},  we have, letting $\sigma=1/2,$
 \begin{eqnarray}\label{esma++}
 		&& \Big<\frac{ \comi{\partial_x
  V(x)}^2  }{\comi{\partial_xV(x)\cdot y }^3} \frac{y^2}{1+y^2}u,~u\Big>_{L^2}+ \Big<\frac{\comi{\partial_xV(x)\cdot y }}{\comi y^4}u,~u\Big>_{L^2}\nonumber \\
  &&\qquad+ \sum_{\ell=1}^n\sum_{j=1}^n\Big\|\sum_{1\leq k\leq n} b_{jk}^\tau(x) \comi{\partial_x V(x)}^{-3/2} y_k  y_\ell  \comi{y}^{-5/2} u \Big\|_{L^2}^2 \\
  &&\leq     C \inner{ \norm{ P u}_{L^2}^2+\norm{u}_{L^2}^2}.\nonumber
 \end{eqnarray}
Moreover observe 
\begin{eqnarray*}
	&& \sum_{\ell=1}^n \norm{\comi{\partial_x V(x)}^{1\over 10} y_\ell \comi{y}^{-1} u }_{L^2}^2\\
	&\leq&   \Big<\frac{ \abs{\partial_x
  V(x)}^2  }{\comi{\partial_xV(x)\cdot y }^3} \frac{y^2}{1+y^2}u,~u\Big>_{L^2}+ \Big<\frac{\comi{\partial_xV(x)\cdot y }}{\comi y^4}u,~u\Big>_{L^2}+\norm{\comi y u}_{L^2}^2,
\end{eqnarray*}
due to the estimate
 \begin{eqnarray*}
	\comi{\partial_x V(x)}^{1\over 5}  \frac{y^2}{1+y^2} &\leq& \inner{\frac{1}{10} \frac{ \comi{\partial_x
  V(x)}^2 }{ \comi{\partial_xV(x)\cdot y}^3}  + \frac{9}{10}   \comi{\partial_x
  V(x) \cdot y}^{1/3}}  \frac{y^2}{1+y^2} \\
  &\leq&  \frac{ \comi{\partial_x
  V(x)}^2 }{ \comi{\partial_xV(x)\cdot y}^3}  \frac{y^2}{1+y^2}   +  \frac{\comi{\partial_x
  V(x) \cdot y} }{\comi y^{4}}+\comi y^2.
\end{eqnarray*}
Then  combining the above inequalities,     \eqref{revel} and \eqref{esma++},   we have  
\begin{eqnarray}\label{keyest+}
&&    \sum_{\ell=1}^n \norm{\comi{\partial_x V(x)}^{1\over 10} y_\ell \comi{y}^{-1} u }_{L^2}^2 \nonumber  \\
  &&   \qquad\qquad\qquad +\sum_{\ell=1}^n\sum_{j=1}^n \Big\|\sum_{1\leq k\leq n} b_{jk}^\tau(x) \comi{\partial_x V(x)}^{-3/2} y_k  y_\ell  \comi{y}^{-5/2} u \Big\|_{L^2}^2\\
  &\leq&   C\Big(\norm{P u}_{L^2}^2+ \norm{u}_{L^2}^2\Big).   \nonumber 	
\end{eqnarray}
In order to obtain a lower bound of the terms on the left hand side of \eqref{keyest+}, we will use  \eqref{comest} with 
\begin{eqnarray*}
	\theta \cdot y =\sum_{\ell=1}^n\comi{\partial_x V(x)}^{1\over 10} y_\ell,\quad v=\comi y^{-1} u;
\end{eqnarray*}
this implies  
\begin{eqnarray*}
   \norm{\comi{\partial_x V(x)}^{1\over 20}  \comi{y}^{-1} u }_{L^2} &\leq & C  \inner{ \norm{\sum_{\ell=1}^n \comi{\partial_x V(x)}^{1\over 10} y_\ell \comi{y}^{-1} u }_{L^2}+\norm{ \partial_{y_k} \comi{y}^{-1} u }_{L^2}}\\
   &\leq & C \sum_{\ell=1}^n   \norm{\comi{\partial_x V(x)}^{1\over 10} y_\ell \comi{y}^{-1} u }_{L^2}+C\Big(\norm{ Pu }_{L^2}+\norm{u }_{L^2}\Big).
\end{eqnarray*}
As a result,  it follows from the above inequalities and \eqref{keyest+} that 
\begin{eqnarray*}
&&     \norm{\comi{\partial_x V(x)}^{1\over 20}   \comi{y}^{-1} u }_{L^2}^2  +\sum_{\ell=1}^n\sum_{j=1}^n \Big\|\sum_{1\leq k\leq n} b_{jk}^\tau(x) \comi{\partial_x V(x)}^{-3/2} y_k  y_\ell  \comi{y}^{-5/2} u \Big\|_{L^2}^2\\
  &\leq&    C\Big(\norm{P u}_{L^2}^2+ \norm{u}_{L^2}^2\Big),
\end{eqnarray*}
and thus,  using again \eqref{keyest+} and repeating the arguments used to prove  \eqref{estimate+},
\begin{eqnarray}\label{+keyest+}
&&     \norm{\comi{\partial_x V(x)}^{1\over 20} u }_{L^2}^2   
 +\sum_{\ell=1}^n\sum_{j=1}^n \Big\|\sum_{1\leq k\leq n} b_{jk}^\tau(x) \comi{\partial_x V(x)}^{-3/2} y_k  y_\ell  \comi{y}^{-5/2} u \Big\|_{L^2}^2\nonumber 	\\
  &\leq&    C\Big(\norm{P u}_{L^2}^2+ \norm{u}_{L^2}^2\Big).   
\end{eqnarray}
Similarly,   for any $1\leq j, \ell\leq n$ we  use \eqref{comest} again with
\begin{eqnarray*}
	\theta \cdot y =\sum_{k=1}^n b_{jk}^\tau(x)\comi{\partial_x V(x)}^{-3/2}y_k, \quad v=y_\ell  \comi{y}^{-5/2} u, 
\end{eqnarray*}
 to obtain
 \begin{eqnarray*}
 &&\sum_{1\leq p\leq n}   \Big\|  \abs{b_{jp}^\tau(x)}^{1/2 } \comi{\partial_x V(x)}^{-3/4}  y_\ell  \comi{y}^{-5/2} u \Big\|_{L^2} \\
 &\leq & C\sum_{1\leq p\leq n} \Big( \Big\|\sum_{1\leq k\leq n} b_{jk}^\tau(x) \comi{\partial_x V(x)}^{-3/2} y_k  y_\ell  \comi{y}^{-5/2} u \Big\|_{L^2}
 +  \norm{\partial_{y_p} y_\ell   \comi{y}^{-5/2} u}_{L^2}\Big)\\
 &\leq & C  \Big\|\sum_{1\leq k\leq n} b_{jk}^\tau(x) \comi{\partial_x V(x)}^{-3/2} y_k  y_\ell  \comi{y}^{-5/2} u \Big\|_{L^2}
 +  C\sum_{1\leq p\leq n} \inner{\norm{\partial_{y_p}  u}_{L^2}+\norm{ u}_{L^2}}\\
 &\leq & C  \Big\|\sum_{1\leq k\leq n} b_{jk}^\tau(x) \comi{\partial_x V(x)}^{-3/2} y_k  y_\ell  \comi{y}^{-5/2} u \Big\|_{L^2}
 +  C  \Big(\norm{P  u}_{L^2}+\norm{ u}_{L^2}\Big).
  \end{eqnarray*}
   This, along with \eqref{+keyest+} and \eqref{keyest+},   yields
 \begin{eqnarray}\label{+keyest+++}
&&      \norm{\comi{\partial_x V(x)}^{1\over 20} u }_{L^2} 
+ \sum_{\ell=1}^n \norm{\comi{\partial_x V(x)}^{1\over 10} y_\ell \comi{y}^{-1} u }_{L^2}\nonumber \\
  &&   \qquad\qquad\qquad +\sum_{\ell=1}^n\sum_{j=1}^n\sum_{p=1}^n  \Big\|  \abs{b_{jp}^\tau(x)}^{1/2 } \comi{\partial_x V(x)}^{-3/4}  y_\ell  \comi{y}^{-5/2} u \Big\|_{L^2}\\
  &\leq&    C\Big(\norm{P u}_{L^2}+ \norm{u}_{L^2}\Big).   \nonumber 	
\end{eqnarray}
Moreover, observing 
\begin{eqnarray*}
	\abs{b_{jp}^\tau(x)}^{1/17}  \comi y^{45/34}\leq \frac{2}{17} \abs{b_{jp}^\tau(x)}^{1/2}\comi{\partial_x V(x)}^{-3/4}+ \frac{15}{17} \comi{\partial_x V(x)}^{1/10}\comi y^{3/2}
\end{eqnarray*}
due to  Young's inequality,  we conclude 
\begin{eqnarray*}
&& \Big\|   \abs{b_{jp}^\tau(x)}^{1/17 }  \comi{y}^{45/34} y_\ell  \comi{y}^{-5/2} u \Big\|_{L^2}\\
 &\leq& C  \Big\|  \abs{b_{jp}^\tau(x)}^{1/2 } \comi{\partial_x V(x)}^{-3/4}  y_\ell  \comi{y}^{-5/2} u \Big\|_{L^2}+C \norm{\comi{\partial_x V(x)}^{1\over 10}  \comi{y}^{3/2} y_\ell \comi{y}^{-5/2} u }_{L^2},
\end{eqnarray*}
that is,
\begin{eqnarray*}
&& \Big\|    \abs{b_{jp}^\tau(x)}^{1\over17 }    y_\ell  \comi{y}^{-{20\over 17}} u \Big\|_{L^2}\\
 &\leq& C  \Big\|  \abs{b_{jp}^\tau(x)}^{1/2 } \comi{\partial_x V(x)}^{-3/4}  y_\ell  \comi{y}^{-5/2} u \Big\|_{L^2}+C \norm{\comi{\partial_x V(x)}^{1\over 10}    y_\ell \comi{y}^{-1} u }_{L^2}.
\end{eqnarray*}
Combing the above estimate and \eqref{+keyest+++},  it follows that 
\begin{eqnarray}\label{esent}
       \norm{\comi{\partial_x V(x)}^{1\over 20} u }_{L^2} 
+ \sum_{\ell=1}^n\sum_{j=1}^n\sum_{p=1}^n  \Big\|    \abs{b_{jp}^\tau(x)}^{1\over17 }    y_\ell  \comi{y}^{-{20\over17}} u \Big\|_{L^2} 
 \leq    C\Big(\norm{P u}_{L^2}+ \norm{u}_{L^2}\Big).  	
\end{eqnarray}
Now we use \eqref{comest} to obtain
\begin{eqnarray*}
	\Big\|    \abs{b_{jp}^\tau(x)}^{1\over34 }      \comi{y}^{-{20\over17}} u \Big\|_{L^2}  	&\leq & C\inner{\Big\|    \abs{b_{jp}^\tau(x)}^{1\over17 }    y_\ell  \comi{y}^{-{20\over17}} u \Big\|_{L^2}+\norm{ \partial_{y_\ell}  \comi{y}^{-{20\over17}} u }_{L^2}}\\
		&\leq &C \Big\|    \abs{b_{jp}^\tau(x)}^{1\over17 }    y_\ell  \comi{y}^{-{20\over17}} u \Big\|_{L^2} +C\inner{\norm{P u }_{L^2}+\norm{  u }_{L^2}},
\end{eqnarray*}
which together with \eqref{esent} gives
\begin{eqnarray*}
	   && \norm{\comi{\partial_x V(x)}^{1\over 20} u }_{L^2} 
+ \sum_{\ell=1}^n\sum_{j=1}^n\sum_{p=1}^n  \Big\|    \abs{b_{jp}^\tau(x)}^{1\over17 }    y_\ell  \comi{y}^{-{20\over17}} u \Big\|_{L^2} +\sum_{j=1}^n\sum_{p=1}^n\Big\|    \abs{b_{jp}^\tau(x)}^{1\over34 }      \comi{y}^{-{20\over17}} u \Big\|_{L^2} \\
 &\leq&    C\Big(\norm{P u}_{L^2}+ \norm{u}_{L^2}\Big),
\end{eqnarray*}
that is, 
\begin{eqnarray*}
	  \norm{\comi{\partial_x V(x)}^{1\over 20} u }_{L^2} 
+  \sum_{j=1}^n\sum_{p=1}^n  \Big\|    \abs{b_{jp}^\tau(x)}^{1\over34 }     \comi{y}^{-{3\over17}} u \Big\|_{L^2} 
 \leq    C\Big(\norm{P u}_{L^2}+ \norm{u}_{L^2}\Big).
\end{eqnarray*}
As a result, using the inequality 
\begin{eqnarray*}
	\abs{b_{jp}^\tau(x)}^{1\over 40}\leq \frac{17}{20}\abs{b_{jp}^\tau(x)}^{1\over 34 }\comi y^{-{3\over 17}}+ \frac{3}{20}\comi y , 
\end{eqnarray*}
we conclude
\begin{eqnarray}\label{1237}
 \norm{\comi{\partial_x V(x)}^{1\over 20} u }_{L^2} 
+  \sum_{j=1}^n\sum_{p=1}^n  \norm{ \abs{b_{jp}^\tau(x)}^{1\over 40}      u }_{L^2}
  &\leq&     C\Big(\norm{P u}_{L^2}+ \norm{\comi y u}_{L^2}\Big).\nonumber
\\
 &\leq&     C\Big(\norm{P u}_{L^2}+ \norm{u}_{L^2}\Big).
 \end{eqnarray}
Recall $A_\tau=B_\tau^TB_\tau$, that is 
\begin{eqnarray*}
	a_{ij}^\tau(x)=\sum_{p=1}^nb_{pi}^\tau(x)b_{pj}^\tau(x).
\end{eqnarray*}
Then   
\begin{eqnarray*}
	\sum_{i=1}^n\sum_{j=1}^n \norm{\abs{a_{ij}^\tau(x)}^{1/80} u}_{L^2} \leq  C \sum_{p=1}^n\sum_{j=1}^n \norm{\abs{b_{pj}^\tau(x)}^{1/40} u}_{L^2}.   
\end{eqnarray*}
Thus, combining \eqref{1237}, we conclude   
\begin{eqnarray*}
	 \norm{\comi{\partial_x V(x)}^{1\over 20} u }_{L^2} 
+  \sum_{i=1}^n\sum_{j=1}^n \norm{\abs{a_{ij}^\tau(x)}^{1/80} u}_{L^2} \leq  C\Big(\norm{P u}_{L^2}+ \norm{ u}_{L^2}\Big).
\end{eqnarray*}
The proof  of Theorem \ref{+Hypo} is thus complete. 
\end{proof}

\end{document}